\documentclass[final,1p,times,number]{elsarticle}
\journal{ }
\usepackage[colorlinks]{hyperref}

\newtheorem{thm}{Theorem}
[section]

\newtheorem{prop}[thm]{Proposition}

\newtheorem{rem}[thm]{Remark}
\newtheorem{fct}[thm]{Fact}
\newtheorem{defn}[thm]{Definition}

\newtheorem{rmk}[thm]{Remark}

\newproof{proof}{Proof}

\newcommand{\gen}[1]{\left\langle#1\right\rangle}
\newcommand{\D}{\delta}

\usepackage{color}
\begin{document}

\title{On Transformations in the Painlev\'e Family.}
\author{Joel Nagloo}
\ead{jnagloo@gc.cuny.edu}
\date{\today}
\address{Graduate Center, Mathematics\\City University of New York\\NY 10016-4309, United States.}
\pagestyle{plain}
\begin{abstract}
In this paper we show that generic Painlev\'e equations from different families are orthogonal. In particular, this means that there are no general Backlund transformations between Painlev\'e equations from the different families $P_I-P_{VI}$.
\end{abstract}

\fntext[fn]{Joel Nagloo was partially supported by the NSF grant CCF-0952591}
\maketitle

\section{Introduction}
As well known, the Painlev\'e equations $P_I-P_{VI}$ are given by the following families of second order algebraic differential equations:
\[
\begin{array}{ll}\medskip
P_{I}:\;\;\;\;\;  & \frac{d^2y}{dt^2}=6y^2+t \\\medskip
P_{II}(\alpha):\;\;\;\;\; & \frac{d^2y}{dt^2}=2y^3+ty+\alpha \\\medskip
P_{III}(\alpha,\beta,\gamma,\delta):\;\;\;\;\; & \frac{d^2y}{dt^2}=\frac{1}{y}\left(\frac{dy}{dt}\right)^2-\frac{1}{t}\frac{dy}{dt}+\frac{1}{t}(\alpha y^2+\beta)+\gamma y^3+\frac{\delta}{y} \\\medskip
P_{IV}(\alpha,\beta):\;\;\;\;\; & \frac{d^2y}{dt^2}=\frac{1}{2y}\left(\frac{dy}{dt}\right)^2+\frac{3}{2}y^3+4ty^2+2(t^2-\alpha)y+\frac{\beta}{y} \\\medskip
P_{V}(\alpha,\beta,\gamma,\delta):\;\;\;\;\; & \frac{d^2y}{dt^2}=\left(\frac{1}{2y}+\frac{1}{y-1}\right)\left(\frac{dy}{dt}\right)^2-\frac{1}{t}\frac{dy}{dt}+\frac{(y-1)^2}{t^2}\left(\alpha y+\frac{\beta}{y}\right)+\gamma\frac{y}{t}\\\medskip
& +\delta\frac{y(y+1)}{y-1}\\ \medskip
P_{VI}(\alpha,\beta,\gamma,\delta):\;\;\;\;\; & \frac{d^2y}{dt^2}=\frac{1}{2}\left(\frac{1}{y}+\frac{1}{y-1}+\frac{1}{y-t}\right)\left(\frac{dy}{dt}\right)^2-\left(\frac{1}{t}+\frac{1}{t-1}+\frac{1}{y-t}\right)\frac{dy}{dt}\\\medskip
 & +\frac{y(y-1)(y-t)}{t^2(t-1)^2}\left(\alpha+\beta\frac{t}{y^2}+\gamma\frac{t-1}{(y-1)^2}+\delta\frac{t(t-1)}{(y-t)^2}\right)
\end{array}
\]
where $\alpha,\beta,\gamma,\delta\in\mathbb{C}$. They were isolated by Painlev\'e and Gambier as those ODE's of the form $y'' = f(y,y',t)$ (where $f$ is rational over ${\mathbb C}$) which have the Painlev\'e property: any local analytic solution extends to a meromorphic solution on the universal cover of $P^{1}({\mathbb C})\setminus S$, where $S$ is the finite set of singularities of the equation (including the point at infinity if necessary). 

An important feature of the families $P_{II}-P_{VI}$  is the existence of the so-called general Backlund transformations that take solutions of a Painlev\'e equation in one family, to solutions of a different Painlev\'e equation in the same family. For example, given a solution $z$ of $P_{II}(\alpha)$, then it is not hard to see that 
\begin{eqnarray*}
S(z)&=&-z\\
T_{+}(z)&=&-z-\frac{\alpha+1/2}{z'+z^2+t/2}\\
T_{-}(z)&=&-z+\frac{\alpha-1/2}{z'-z^2-t/2}
\end{eqnarray*}
are solutions of $P_{II}(-\alpha)$, $P_{II}(\alpha+1)$ and $P_{II}(\alpha-1)$ respectively. The transformations $S$, $T_{+}$ and $T_{-}$ generate a group $G$ that provides a representation of the affine Weyl group of type $\tilde{A}_1$ and it is known that similar observations can be made about the other equations $P_{III}-P_{VI}$. The general Backlund transformations have played a crucial role in the classification of algebraic and classical solutions of the Painlev\'e equations.  

If we write $X_{II}(\alpha)$ for the solution set of $P_{II}(\alpha)$, in a differentially closed field, then we see that $S$, $T_{+}$ and $T_{-}$ give rise to infinite ``finite-to-finite'' differential relations (indeed one-to-one) between $X_{II}(\alpha)$ and $X_{II}(-\alpha)$, $X_{II}(\alpha+1)$ and $X_{II}(\alpha-1)$ respectively:
\[
\begin{array}{lll}\medskip
R&=&\{(z_1,z_2)\in X_{II}(\alpha)\times X_{II}(-\alpha):z_2=S(z_1)\}\\\medskip
R_+&=&\{(z_1,z_2)\in X_{II}(\alpha)\times X_{II}(\alpha+1):z_2=T_{+}(z_1)\}\\\medskip
R_-&=&\{(z_1,z_2)\in X_{II}(\alpha)\times X_{II}(\alpha-1):z_2=T_{-}(z_1)\}
\end{array}
\]From a model theoretic standpoint this simply mean that for $\alpha\not\in 1/2+\mathbb{Z}$, that is those $\alpha$'s where $X_{II}(\alpha)$ is strongly minimal (see Definition \ref{defnUm} below), $X_{II}(\alpha)$ is nonorthogonal to $X_{II}(-\alpha)$, $X_{II}(\alpha+1)$ and $X_{II}(\alpha-1)$. Nonorthogonality is a fundamental notion in model theory and finding out whether two Painlev\'e equations are nonorthogonal is a very natural problem to tackle. Of course, the existence of the other Backlund transformations implies that there are equations within the other families $P_{III}-P_{VI}$ that are nonorthogonal.

In this paper we show that generic Painlev\'e equations, that is those with parameters in general positions, from different families are orthogonal, answering a question of P. Boalch \cite{private}. In particular, this means that there are no general Backlund transformations between Painlev\'e equations from the different families $P_I-P_{VI}$. 

The paper is organized as follows. We set up our notations and introduce the basic notions from model theory and differential algebra in Section 2. We then, in Section 3, give a survey of the results around the (model theoretic) classification of the Painlev\'e equations as this will be needed in the proofs of our main results. Nonorthogonality in its various form is introduced and studied in Section 4 and finally Section 5 is where our main results (Proposition 5.1-5.6) are proved.


\section{Preliminaries}
Throughout this paper, we fix a saturated model $\mathcal U = ({\mathcal U},+,-,\cdot,0,1,\partial)$ of $DCF_{0}$, the theory of differentially closed fields of characteristic $0$ with a single derivation in the language $L_{\partial}$=($+,-,\cdot,0,1,\partial$) of differential rings. We assume that the cardinality of $\mathcal U$ is the continuum and so we can and will identify the field of constants of $\mathcal U$ with $\mathbb{C}$, the field of complex numbers.  We also fix once and for all an element $t$ of $\mathcal{U}$ having the property that $\partial(t)=1$.

As well known, $DCF_0$ is complete, has quantifier elimination and is $\omega$-stable (cf. \cite{Marker}). Although not explicit, these properties play an important role in this paper. For a field $K$, $K^{alg}$ will denote its algebraic closure in the usual algebraic sense and by a definable set we mean a finite Boolean combination of affine differential algebraic varieties (or Kolchin closed sets). If a definable set $X$ in $\mathcal{U}^n$ is defined with parameters from a differential subfield $K$ of $\mathcal{U}$, we will say that $X$ is defined over $K$ and furthermore if $\overline{y}=(y_1,\ldots,y_n)$ is a tuple from ${\mathcal U}$, $K\langle\overline{y} \rangle$ will denote the differential field $K(\overline{y},\overline{y}',\overline{y}'',\ldots)$, where $\overline{y}'$ is short for $\partial(\overline{y})=(\partial(y_1),\ldots,\partial(y_n))$. Finally, as usual, given a tuple $\overline{z}$ from ${\mathcal U}$, $\overline{z}\in K\langle\overline{y} \rangle^{alg}$ means that the coordinates of $\overline{z}$ are in $K\langle\overline{y} \rangle^{alg}$.  
\begin{defn}\label{defnUm}
An algebraic differential equation of the form $y''= f(y,y',t)$, where $f$ is rational over $\mathbb C$, is said to be {\em irreducible with respect to classical functions} if
\begin{enumerate}
\item for any differential subfield $K$ of $\mathcal U$ which is finitely generated over $\mathbb{C}(t)$, and $y$ solution, either $y\in K^{alg}$ or $tr.deg(K\langle y \rangle/K) = 2$
\item there are no solutions in $\mathbb{C}(t)^{alg}$
\end{enumerate}
\end{defn}
This notion was introduce by Umemura \cite{Umemura1} (see also the Appendix of \cite{NagPil}) following ideas of Painlev\'e in his study of the first Painlev\'e equation. As we will see later, one of the big advances in the Painlev\'e theory is the full classification of all Painlev\'e equations that are irreducible with respect to classical functions.

If we denote by $X\subset\mathcal{U}$ the solution set of $y''= f(y,y',t)$, then condition 1 in Definition \ref{defnUm} is equivalent to $X$ being {\em strongly minimal}; namely, $X$ is infinite and has no infinite co-infinite definable subsets. 
Strongly minimal sets are of fundamental importance in the study of differentially closed fields from the model theoretic standpoint and in \cite{NagPil} one can find a very detailed summary of the main results around the ``geometry" of strongly minimal sets. In particular, in that paper one can find an explanation as to why there essentially are only three kinds of strongly minimal sets. When looking at the generic Painlev\'e equations only geometrically trivial ones are relevant.

\begin{defn} Let $X\subset {\mathcal U}^{n}$ be strongly minimal. We say that $X$ is {\em geometrically trivial} if for any countable differential field $K$ over which $X$ is defined, and for any $y_{1},..,y_{\ell}\in X$, denoting $\tilde{y}_i$ the tuple given by $y_i$ together with all its derivatives, if $(\tilde{y}_1,\ldots,\tilde{y}_{\ell})$ is algebraically dependent over $K$ then for some $i<j$, $\tilde{y}_{i}, \tilde{y}_{j}$ are algebraically dependent over $K$.
\end{defn}
Geometric triviality limits the possible complexity of the structure of a strongly minimal set. However, given a geometrically trivial strongly minimal set, there still is the problem of determining its precise structure.

\begin{defn} Suppose $X\subset {\mathcal U}^{n}$ is a geometrically trivial strongly minimal set defined over some differential field $K$.
\begin{enumerate}
\item $X$ is said to be {\em strictly disintegrated over $K$} if for any $y\in X$,  we have that $X\cap K\gen{y}^{alg}=\{y\}$.
\item $X$ is said to be {\em $\omega$-categorical} if for any $y\in X$,  we have that $X\cap K\gen{y}^{alg}$ is finite.
\end{enumerate}
\end{defn}
So strict disintegratedness (resp. $\omega$-categoricity) means that there is no (resp. very little) structure. It had been long conjectured that all geometrically trivial sets in $DCF_0$ are $\omega$-categorical. However recently, Freitag and Scanlon have found examples of geometrically trivial sets which have rich binary structures (\cite{FreScan}). As we shall now see, the same is not true of the Painlev\'e equations. 

\section{The classification of the Painlev\'e equations}
For each of the six families of Painlev\'e equations we will now give a summary of the various results around the classification of their solution sets. This will play an important role in the proofs of the main results. We direct the reader to the author's thesis \cite{Nag1} for more details about the results below. We will say that an equation in one of the Painlev\'e families is ``generic'' if the corresponding complex parameters are mutually generic, that is if they form an algebraically independent tuple of transcendental complex numbers.

 \subsection{The first Painlev\'e equation $P_{I}$}
\begin{fct}[\cite{Nishioka},\cite{Umemura1}]\label{P1} Let $X_{I}$ be the solution set in $\mathcal{U}$ of $P_{I}$
\[ y''=6y^2+t .\]Then
\begin{enumerate}
\item $X_{I}$ is strongly minimal and has no solution in $\mathbb{C}(t)^{alg}$. In other words, $P_{I}$ is irreducible with respect to classical functions.
\item $X_{I}$ is geometrically trivial and strictly disintegrated over $\mathbb{C}(t)$.
\end{enumerate}
\end{fct}

\subsection{The second Painlev\'e equation $P_{II}$}

\begin{fct}[\cite{Murata},\cite{NagPil},\cite{NagPil1},\cite{Nag},\cite{Umemura2}]\label{P2} Let $\alpha\in\mathbb{C}$ and let $X_{II}(\alpha)$ be the solution set in $\mathcal{U}$ of $P_{II}(\alpha)$
\[ y''=2y^3+ty+\alpha.\]Then
\begin{enumerate}
\item $X_{II}(\alpha)$ is strongly minimal if and only if $\alpha\not\in 1/2+\mathbb{Z}$.
\item $P_{II}(\alpha)$ has a solution in $\mathbb{C}(t)^{alg}$ if and only if $\alpha\in\mathbb{Z}$. Furthermore, the solution when it exists is unique.
\item For any $\alpha\not\in 1/2+\mathbb{Z}$, $X_{II}(\alpha)$ is geometrically trivial.
\item For generic $\alpha$ (i.e. $\alpha\not\in\mathbb{Q}^{alg}$), $X_{II}(\alpha)$ is strictly disintegrated over $\mathbb{C}(t)$.
\end{enumerate}
\end{fct}

\subsection{The third Painlev\'e equation $P_{III}$}
For $P_{III}$, Okamoto \cite{Okam4} (see also \cite{Umemura3}) shows that when working with generic parameters, it is enough to rewrite the equation as a 2-parameter family. 
\begin{fct}[\cite{Murata2},\cite{NagPil},\cite{NagPil1},\cite{Umemura3}]\label{P3} Let $v_{1},v_{2}\in\mathbb{C}$ and let $X_{III}(v_1,v_2)$ be the solution sets in $\mathcal{U}$ of $P_{III}(v_1,v_2)$ 
\[
y''=\frac{1}{y}(y')^2-\frac{1}{t}y'+\frac{4}{t}(v_1+1-v_2 y^2)+4y^3-\frac{4}{y}.
\]
Then
\begin{enumerate}
\item $X_{III}({v_{1},v_{2}})$ is strongly minimal if and only if $v_{1}+ v_{2}\notin 2{\mathbb Z}$ and $v_{1}-v_{2}\notin 2{\mathbb Z}$. 
\item$P_{III}({v_1,v_2})$ has algebraic solutions if and only if there exists an integer $n$ such that $v_2-v_1-1=2n$ or $v_2+v_1+1=2n$.
\item If $P_{III}({v_1,v_2})$ has algebraic solutions, then the number of algebraic solutions is two or four. $P_{III}({v_1,v_2})$ has four algebraic solutions if and only if there exist two integers $n$ and $m$ such that $v_2-v_1-1=2n$ and $v_2+v_1+1=2m$.
\item If $v_1,v_2$ are mutually generic, then $X_{III}(v_1,v_2)$ is geometrically trivial and $\omega$-categorical.
\end{enumerate}
\end{fct} 
The proof of Fact \ref{P3}(4) given in \cite{NagPil1} tells us a little more:
\begin{rem}\label{RP3}
Let $v_1,v_2$ be generic and let $y\in X_{III}(v_1,v_2)$. Then the cardinality of $X_{III}(v_1,v_2)\cap \mathbb{C}(t)\gen{y}^{alg}$ is at most 2  (including $y$ itself).
\end{rem}

\subsection{The fourth Painlev\'e equation $P_{IV}$}

\begin{fct}[\cite{Murata},\cite{NagPil},\cite{NagPil1},\cite{Umemura2}]\label{P4}  Let $\alpha,\beta\in\mathbb{C}$ and let $X_{IV}(\alpha,\beta)$ be the solution set in $\mathcal{U}$ of $P_{IV}(\alpha,\beta)$
\[y''=\frac{1}{2y}(y')^2+\frac{3}{2}y^3+4ty^2+2(t^2-\alpha)y+\frac{\beta}{y}.\] Then
\begin{enumerate}
\item $X_{IV}(\alpha,\beta)$ is strongly minimal if and only if $v_1-v_2\not\in\mathbb{Z}$ or $v_2-v_3\not\in\mathbb{Z}$ or $v_3-v_1\not\in\mathbb{Z}$, for any $v_1,v_2,v_3\in\mathbb{C}$ with $v_1+v_2+v_3=0$ and such that $\alpha = 3v_{3} + 1$ and $\beta = -2(v_{2}-v_{1})^{2}$.
\item $P_{IV}$ has algebraic solutions if and only if $\alpha,\beta$ satisfy one of the following conditions:
\begin{enumerate}[(i)]
\item $\alpha=n_1$ and $\beta=-2(1+2n_2-n_1)^2$, where $n_1,n_2\in\mathbb{Z}$;
\item $\alpha=n_1$, $\beta=-\frac{2}{9}(6n_2-3n_1+1)^2$, where $n_1,n_2\in\mathbb{Z}$.
\end{enumerate} Furthermore the algebraic solutions for these parameters are unique.
\item If $\alpha,\beta$ are mutually generic, then $X_{IV}(\alpha,\beta)$ is geometrically trivial and strictly disintegrated over $\mathbb{C}(t)$.
\end{enumerate}
\end{fct}

\subsection{The fifth Painlev\'e equation $P_{V}$}

The situation for $P_{V}$ is similar to the case of $P_{III}$ above, in that, in the light of certain transformations preserving the equation and by fixing $\delta=-1/2$, $P_{V}$ can be written as a  $3$-parameter family. The analysis is carried out by Okamoto\cite{Okam2} and then Watanabe \cite{Watanabe}

\begin{fct}[\cite{Kitaev},\cite{NagPil},\cite{NagPil1},\cite{Watanabe}]\label{P5}  Let $\alpha,\beta,\gamma\in\mathbb{C}$ and let $X_{V}(\alpha,\beta,\gamma)$ be the solution set in $\mathcal{U}$ of $P_{V}(\alpha,\beta,\gamma,-1/2)$
\[y''=\left(\frac{1}{2y}+\frac{1}{y-1}\right)( y')^2-\frac{1}{t}y'+\frac{(y-1)^2}{t^2}\left(\alpha y+\frac{\beta}{y}\right)+\gamma\frac{y}{t}-\frac{1}{2}\frac{y(y+1)}{y-1},\] Then
\begin{enumerate}
\item $X_{V}(\alpha,\beta,\gamma)$ is strongly minimal if and only if $v_i-v_j\not\in\mathbb{Z}$ ($i\neq j$), for any $v_1,v_2,v_3,v_4\in\mathbb{C}$ with $v_1+v_2+v_3+v_4=0$ and such that $\alpha = \frac{1}{2}(v_3-v_4)^2$, $\beta =-\frac{1}{2}(v_2-v_1)^2$ and $\gamma = 2v_1+2v_2-1$.
\item  Then $P_{V}(\alpha,\beta,\gamma,-1/2)$ has an algebraic solution if and only if one of the following holds with $m,n\in\mathbb{Z}$:
\begin{enumerate}[(i)]
\item $\alpha=\frac{1}{2}(m+\gamma)^2$ and $\beta=-\frac{1}{2}n^2$ where $n>0$, $m+n$ is odd, and $\alpha\neq 0$ when $|m|<n$;
\item $\alpha=\frac{1}{2}n^2$ and $\beta=-\frac{1}{2}(m+\gamma)^2$ where $n>0$, $m+n$ is odd, and $\beta\neq 0$ when $|m|<n$;
\item $\alpha=\frac{1}{2}a^2$, $\beta=-\frac{1}{2}(a+n)^2$ and $\gamma=m$, where $m+n$ is even and $a$ arbitrary;
\item $\alpha=\frac{1}{8}(2m+1)^2$, $\beta=-\frac{1}{8}(2n+1)^2$ and $\gamma\not\in\mathbb{Z}$.
\end{enumerate}
\item If $\alpha,\beta,\gamma$ are mutually generic, then $X_{V}(\alpha,\beta,\gamma)$ is geometrically trivial and strictly disintegrated over $\mathbb{C}(t)$.
\end{enumerate}
\end{fct}
\begin{rmk}\label{RP5}{\color{white}kj}
\begin{enumerate}[(i)]
\item In case 2(iv) the algebraic solution is unique. This is also true for most of the other cases (see \cite{Kitaev}) except for:
\item In case 2(i) or (ii), if $\gamma\in\mathbb{Z}$ then there are at most two algebraic solutions. Specifically, if $\alpha\beta\neq 0$, there are exactly two; otherwise there is only one.
\end{enumerate}
\end{rmk}

\subsection{The sixth Painlev\'e equation $P_{VI}$}

When looking at $P_{VI}$, it is more convenient to work with its hamiltonian form.
\begin{fct}[\cite{Lisovyy},\cite{NagPil},\cite{NagPil1},\cite{Watanabe}] \label{P6} Let $\overline{\alpha}=(\alpha_0,\alpha_1,\alpha_2,\alpha_3,\alpha_4)\in\mathbb{C}^5$ be such that  such that $\alpha_0+\alpha_1+2\alpha_2+\alpha_3+\alpha_4=1$ and let $X_{VI}(\overline{\alpha})$ be the solution set in $\mathcal{U}$ of 
\[S_{VI}(\overline{\alpha})\left\{
\begin{array}{rll}
y' &=& \frac{1}{t(t-1)}(2xy(y-1)(y-t)-\{\alpha_4(y-1)(y-t)+\alpha_3y(y-t)\\
& & +(\alpha_0-1)y(y-1)\})\\
x'&=& \frac{1}{t(t-1)}(-x^2(3y^2-2(1+t)y+t)+x\{2(\alpha_0+\alpha_3+\alpha_4-1)y\\
& &-\alpha_4(1+t)-\alpha_3t-\alpha_0+1\}-\alpha_2(\alpha_1+\alpha_2))
\end{array}\right.\]
Let $\mathcal{M}$ be the union of the hyperplanes given by $\{\alpha_{i} = n\in\mathbb{Z}\}$ for $i=0,1,3,4$ and $\{\alpha_0\pm\alpha_1\pm\alpha_3\pm\alpha_4-1=m\in 2\mathbb{Z}\}$.
\begin{enumerate}
\item If $\alpha_1,\alpha_3,\alpha_4\in\mathbb{C}$ are mutually generic, then for  $\alpha_0\in 2\mathbb{Z}$, $X_{VI}(\overline{\alpha})$ is not strongly minimal.
\item If $\alpha_{0}, \alpha_{1}, \alpha_{3}, \alpha_{4}$ are mutually generic then $X_{VI}(\overline{\alpha})$ is is geometrically trivial and $\omega$-categorical.
\end{enumerate}
\end{fct}
We do not give the results concerning the full classification of algebraic solutions as this is beyond the scope of this paper (cf. \cite{Lisovyy}). However it is well known that for $\alpha_0=\alpha_1=\alpha_3=\alpha_4=0$, $X_{VI}$ has infinitely many algebraic solutions over $\mathbb{C}(t)$. Indeed, $X_{VI}(0,0,1/2,0,0)$ can be identified with the smallest Zariski dense definable subgroup of the elliptic curve $E_t:\; y^2=x(x-1)(x-t)$ (i.e.  $X_{VI}(0,0,1/2,0,0)$ is what is called the Manin kernel of $E_t$). Algebraic solutions then corresponds to the torsion points $E_t^{tor}$. Furthermore, by applying Backlund transformations (see \cite{Okam1}) one has the following:
\begin{fct}
For $\alpha_0,\alpha_1,\alpha_3,\alpha_4\in 1/2+\mathbb{Z}$, $X_{VI}(\overline{\alpha})$ has infinitely many algebraic solutions over $\mathbb{C}(t)$.
\end{fct}

\subsection{Conclusion}
\noindent As we shall see in Section 5, the key points from this section that will play an important role in our strategy for the proofs of the main results are:
\begin{enumerate}
\item All the generic Painlev\'e equations are irreducible with respect to classical functions, geometrically trivial and $\omega$-categorical (of course in most cases strictly disintegrated).
\item For each family, there are exceptional sets of complex parameters (e.g. $\mathbb{Z}$ and $1/2+\mathbb{Z}$ for $P_{II}$) where either algebraic solutions exist or strong minimality does not hold.
\end{enumerate}
\begin{rmk}
It is also crucial to note that for any of the families $P_{II}-P_{VI}$, if for some choice of parameters the equation is not strongly minimal, then this is witnessed by the existence of an order 1 algebraic differential subvariety. For example for $P_{II}(-1/2)$ it is not hard to see that any solution of $y'=-y^2-t/2$ is also a solution of $P_{II}(-1/2)$. These differential subvarieties are usually called the Riccati solutions of the Painlev\'e equations.
\end{rmk}

\section{Nonorthogonality}
Let us now introduce the most important notion used in this paper: nonorthogonality. We restrict ourselves to strongly minimal sets (or types) as this is all we need for the Painlev\'e equations. We direct the reader to \cite{Pillay-book} for the more general context.
\begin{defn} Suppose $X$ and $Y$ are strongly minimal sets and denote by $\pi_1:X\times Y\rightarrow X$ and $\pi_2:X\times Y\rightarrow Y$ the projections to $X$ and $Y$ respectively. We say that $X$ and $Y$ are {\em nonorthogonal} if there is some infinite definable relation $R\subset X\times Y$ such that ${\pi_1}_{\restriction R}$ and ${\pi_2}_{\restriction R}$ are finite-to-one functions.  
\end{defn}
Nonorthogonality is an equivalence relations on strongly minimal sets and given a family of strongly minimal sets, it is very natural to ask about nonorthogonalities within that family. Of course in the case of the Painlev\'e equations, the Backlund transformations are very relevant. For example, for $X_{II}(\alpha)$ the bijection $T_{+}:X_{II}(\alpha)\rightarrow X_{II}(\alpha+1)$
\[T_{+}(z)=-z-\frac{\alpha+1/2}{z'+z^2+t/2}\]
shows that for $\alpha\not\in 1/2+\mathbb{Z}$, $X_{II}(\alpha)$ is nonorthogonal to $X_{II}(\alpha+1)$. 

Moreover, $X_{II}(\alpha)$ and $X_{II}(\alpha+1)$ are nonorthogonal in a very special way, namely one does not require extra parameters (other that $\alpha$ and $t$) to witness that they are nonorthogonal. It turns out that this is no coincidence but first we need a definition.
\begin{defn}
Two strongly minimal sets $X$ and $Y$ (both defined over some differential field $K$ say) are said to be {\em  non weakly orthogonal} if they are nonorthogonal, that is there is an infinite finite-to-finite relation $R\subseteq X\times Y$, and the formula defining $R$ has no parameters in $\mathcal{U}\setminus K^{alg}$.
\end{defn}
\begin{fct}[{\bf \cite{Pillay-book}, Corollary 2.5.5}]\label{weaklyort}
Let $X$ and $Y$ be a geometrically trivial strongly minimal sets (both defined over some differential field $K$). Assume that they are nonorthogonal. Then they are non weakly orthogonal.
\end{fct}
So in particular if two generic members of the Painlev\'e family are nonorthogonal, then they are non weakly orthogonal. We now proceed to show that distinct generic Painlev\'e equations from any of the families are orthogonal.

\section{On Transformations in the Painlev\'e Family.}
In this section, $X_{I}$, $X_{II}(\alpha)$, $X_{III}({v_{1},v_{2}})$, $X_{IV}(\alpha,\beta)$, $X_{V}(\alpha,\beta,\gamma)$ and $X_{VI}(\overline{\alpha})$ will denote the sets of solution of the Painlev\'e equations as given to us in Section 3. Also as we have seen in Fact \ref{P4} and \ref{P5}, for $P_{IV}(\alpha,\beta)$ and $P_{V}(\alpha,\beta,\gamma)$ by abuse of notation, we will sometime write $X_{IV}(\alpha,\beta)$ and $X_{V}(\alpha,\beta,\gamma)$ as $X_{IV}(v_{1},v_{2},v_3)$ and $X_{V}(v_{1},v_{2},v_{3},v_4)$ respectively. For example for $P_{IV}(\alpha,\beta)$, the notation $X_{IV}(v_{1},v_{2},v_3)$ simply mean that (as in Section 3.4) we choose $v_1,v_2,v_3\in\mathbb{C}$ such that $v_1+v_2+v_3=0$, $\alpha = 3v_{3} + 1$ and $\beta = -2(v_{2}-v_{1})^{2}$.

\subsection{The special case of $P_I$}
We show in section that $X_{I}$ is orthogonal to all the other generic Painlev\'e equations. We give a very detailed proof of the first proposition as this will be a model for many of the other cases.
\begin{prop}\label{Ort1}
$X_{I}$ is orthogonal to any generic $X_{II}(\alpha)$.
\end{prop}
\begin{proof}
Let $\alpha\in\mathbb{C}$ be generic and for contradiction, suppose that $X_{I}$ is non weakly orthogonal to $X_{II}(\alpha)$ (this suffices by Fact \ref{weaklyort}). By definition, this means that there exists a finite-to-finite definable relation $R\subseteq X_{I}\times X_{II}(\alpha)$ and that $R$ is defined over $\mathbb{Q}(\alpha,t)^{alg}$. Since both $X_{I}$ and $X_{II}(\alpha)$ are strictly disintegrated (by Fact \ref{P1} and Fact \ref{P2}) and they have no elements in $\mathbb{C}(t)^{alg}$, it is not hard to see that $R$ is the graph of a bijection between $X_{I}$ and $X_{II}(\alpha)$. Let us suppose that $R$ is defined by $\varphi(x,y,\alpha,t)$ and let $\sigma(u,v)$ be the $L_{\D}$ formula $\forall x\exists^{=1}y\varphi(x,y,u,v)\wedge \forall y\exists^{=1}x\varphi(x,y,u,v)$. So $\mathcal{U}\models\sigma(\alpha,t)$ and by construction, $\mathcal{U}\models\sigma(\tilde{\alpha},t)$ implies that $\tilde{\alpha}\in\mathbb{C}$ and $X_{I}$ is in bijection with $X_{II}(\tilde{\alpha})$.\medskip

As $\mathbb{C}$ is strongly minimal and $\alpha$ generic, $\sigma(\tilde{\alpha},t)$ is true for all but finitely many $\tilde{\alpha}\in\mathbb{C}$. In particular for some $\alpha_0\in 1/2+\mathbb{Z}$, we have that $\mathcal{U}\models\sigma(\alpha_0,t)$ and hence $X_{I}$ is in bijection with $X_{II}(\alpha_0)$. By Fact \ref{P2} $X_{II}(\alpha_0)$ is not strongly minimal (and indeed contains an order $1$ definable subset). This is impossible.
\end{proof}
Similarly one has the following
\begin{prop}
$X_{I}$ is orthogonal to the generic $X_{III}({v_{1},v_{2}})$, $X_{IV}(v_{1},v_{2},v_3)$, $X_{V}(v_{1},v_{2},v_{3},v_4)$ and $X_{VI}(\overline{\alpha})$.
\end{prop}
\begin{proof}
In the case of the generic $X_{IV}(v_{1},v_{2},v_3)$ and $X_{V}(v_{1},v_{2},v_{3},v_4)$ the same proof as the one given above works. One only need to replace $1/2+\mathbb{Z}$ by the appropriate exceptional sets. For $X_{III}(v_{1},v_{2})$ and $X_{VI}(\overline{\alpha})$ the only other slight modification is to replace the definable bijection $R$ by one-to-finite maps (as the sets are only $\omega$-categorical). But this does not pose any real problems.
\end{proof}

\subsection{Orthogonality in the remaining cases}
Although the idea of the proofs are the same, when trying to prove for example that generic $X_{II}(\alpha)$ is orthogonal to generic $X_{III}(v_1,v_2)$ one needs to be more careful as the parameters are not necessarily assumed to be mutually generic. We start again with an easy case.
\begin{prop}
The generic $X_{II}(\alpha)$ is orthogonal to the generic $X_{V}(v_{1},v_{2}$, $v_{3},v_4)$.
\end{prop}
\begin{proof}Assume for contradiction the $X_{II}(\alpha)$ is non weakly orthogonal to the generic $X_{V}(v_{1},v_{2}$, $v_{3},v_4)$. Recall that $v_4=-(v_1+v_2+v_3)$ and so for notational simplicity we write $X_{V}(v_{1},v_{2}$, $v_{3},-(v_1+v_2+v_3))$ as $X(v_{1},v_{2}$, $v_{3})$.\\
{\em Case (i): $\alpha$, $v_{1}$ $v_{2}$ and $v_{3}$ are mutually generic.} Then one uses the same proof as that of Proposition \ref{Ort1}.\\
\noindent {\em Case (ii): $\alpha\in\mathbb{Q}(v_{1},v_{2},v_{3})^{alg}$.} For contradiction, suppose that $X_{II}(\alpha)$ is non weakly orthogonal to $X(v_{1},v_{2},v_{3})$. So we have as before a formula $\sigma(u,w_1,w_2,w_3,x)$ such that $\mathcal{U}\models\sigma(\alpha,v_{1},v_{2},v_{3},t)$ and this expresses that $X_{II}(\alpha)$ is in bijection with $X(v_{1},v_{2},v_{3})$. We then have to consider three sub-cases:\\
\noindent {\em Sub-case (i): $\alpha\not\in\mathbb{Q}(v_{i},v_{j})^{alg}$ for any $i\neq j$.} All we have to do is quantify over $v_3$ say, that is we use the fact that $\mathcal{U}\models\exists v_{3}\sigma(\alpha,v_{1},v_{2},v_{3},t)$. As $\alpha\not\in\mathbb{Q}(v_{1},v_{2})^{alg}$ and $\mathbb{C}$ is strongly minimal we can as before find $\tilde{v}_1\in v_2+\mathbb{Z}$ such that $\mathcal{U}\models\exists v_{3}\sigma(\alpha,\tilde{v}_1,v_{2},v_{3},t)$. Choosing any $\tilde{v}_3\in\mathbb{C}$ witnessing this, we get a bijection between $X_{II}(\alpha)$ and $X(\tilde{v}_1,v_{2},\tilde{v}_3)$ (but again $X_{V}(\tilde{v}_1,v_{2},\tilde{v}_3,-(\tilde{v}_1+v_{2}+\tilde{v}_3))$ is not strongly minimal by Fact \ref{P5}). This gives our desired contradiction.\\
\noindent {\em Sub-case (ii): $\alpha\in\mathbb{Q}(v_{1},v_{2})^{alg}$ say but not in $\mathbb{Q}(v_{1})^{alg}$ and $\mathbb{Q}(v_{2})^{alg}$.} First note that as $v_{1}$, $v_{2}$ and $v_{3}$ are mutually generic we have that $\alpha$ is not in $\mathbb{Q}(v_{3})^{alg}$, $\mathbb{Q}(v_{1},v_{3})^{alg}$ or $\mathbb{Q}(v_{2},v_{3})^{alg}$. So this time we quantify over $v_2$ say and look at $\mathcal{U}\models\exists v_{2}\sigma(\alpha,v_{1},v_{2},v_{3},t)$. This allows us again to find $\tilde{v}_1\in v_3+\mathbb{Z}$ such that $\mathcal{U}\models\exists v_{2}\sigma(\alpha,\tilde{v}_1,v_{2},v_{3},t)$. Finally any $\tilde{v}_2\in\mathbb{C}$ making this true leads us to our contradiction.\\
\noindent {\em Sub-case (iii): $\alpha\in\mathbb{Q}(v_{1})^{alg}$ say.} This time all we have to do is to quantify over $v_1$ and the same argument works.
\end{proof}
For the other cases, we need to change a little bit our strategy. We will this time use the results on algebraic solutions of the Painlev\'e equations.
\begin{prop}\label{algebraicproof1}
The generic $X_{II}(\alpha)$ is orthogonal to the generic $X_{III}(v_{1},v_{2})$.
\end{prop}
\begin{proof}
For contradiction, suppose that $X_{II}(\alpha)$ is non weakly orthogonal to $X_{III}(v_{1},v_{2})$. As before, we get a formula $\sigma(u,w_1,w_2,x)$ such that $\mathcal{U}\models\sigma(\alpha,v_{1},v_{2},t)$ and this expresses that there is a ``1 to $\leq$2" map between $X_{II}(\alpha)$ and $X_{III}(v_{1},v_{2})$. This follows from strict disintegratedness of $X_{II}(\alpha)$ and $\omega$-categoricity of $X_{III}(v_{1},v_{2})$ (more precisely Remark \ref{RP3} gives that for any  $y\in X_{III}(v_{1},v_{2})$  the cardinality of $X_{III}(v_{1},v_{2}) \cap\mathbb{C}(t)\gen{y}^{alg}$ is at most 2).\medskip

We then quantify over $\alpha$, that is use that $\mathcal{U}\models\exists\alpha\sigma(\alpha,v_{1},v_{2},t)$. As $v_1$ and $v_2$ are mutually generic, we can first choose $\tilde{v}_1\in (v_2-1)+2\mathbb{Z}$ (and have $\mathcal{U}\models\exists\alpha\sigma(\alpha,\tilde{v}_{1},v_{2},t)$) and then $\tilde{v}_2\in\mathbb{Z}$ so that $\mathcal{U}\models\exists\alpha\sigma(\alpha,\hat{v}_{1},\tilde{v}_{2},t)$ where $\hat{v}_1=\tilde{v}_{2}-1+2m$ for some $m\in\mathbb{Z}$. Choosing any $\tilde{\alpha}\in\mathbb{C}$ witnessing this, we have that there is a ``1 to $\leq$2" map between $X_{II}(\tilde{\alpha})$ and $X_{III}(\hat{v}_{1},\tilde{v}_{2})$. By Fact \ref{P1}, $Y_{II}(\tilde{\alpha})$ contain at most 1 algebraic solution whereas by Fact \ref{P3}(3), $X_{III}(\hat{v}_{1},\tilde{v}_{2})$ has exactly 4 algebraic solutions. So we get a contradiction.
\end{proof}
Similarly one has the following
\begin{prop}{\color{white} dF}
\begin{enumerate}[(i)]
\item The generic $X_{III}(v_{1},v_{2})$ is orthogonal to the generic $X_{IV}(\alpha,\beta)$.
\item The generic $X_{IV}(\alpha,\beta)$ is orthogonal to the generic $X_{V}(\alpha_1,\beta_1,\gamma_1)$.
\end{enumerate}
\end{prop}
\begin{proof}
For (i) the exact same proof as Proposition \ref{algebraicproof1}, where one just needs to quantify over $\alpha$ and $\beta$, works.\\
For (ii) Arguing by contradiction, we have $\mathcal{U}\models\sigma(\alpha,\beta,\alpha_1,\beta_1,\gamma_1,t)$ witnessing that there is a bijection between $X_{IV}(\alpha,\beta)$ and  $X_{V}(\alpha_1,\beta_1,\gamma_1)$. Again by quantifying over $\alpha$ and $\beta$ and moving $\alpha_1,\beta_1$ and $\gamma_1$ into an appropriate set where $X_{V}$ has 2 algebraic solutions (as given to us by Remark \ref{RP5}(ii)), we are done as $X_{IV}(\alpha,\beta)$ can only have at most 1 algebraic solution.
\end{proof}
Finally we look at orthogonality to the sixth Painlev\'e equation. 
\begin{prop}
The generic $X_{VI}(\overline{\alpha})$ is orthogonal to the generic $X_{II}(\alpha)$, $X_{III}({v_{1},v_{2}})$, $X_{IV}(\alpha,\beta)$ and $X_{V}(\alpha,\beta,\gamma)$. 
\end{prop}
\begin{proof}
One just uses the same trick as the proof of Proposition \ref{algebraicproof1}: We want to prove that the generic $X_{VI}(\overline{\alpha})$ is orthogonal to $X(\overline{v})$, where $X(\overline{v})$ is any of the above generic sets.

Arguing by contradiction we get $\mathcal{U}\models\sigma(\overline{v},\alpha_0,\alpha_1,\alpha_3,\alpha_4,t)$ witnessing that there is a finite-to-finite map between $X_{VI}(\overline{\alpha})$ and $X(\overline{v})$. Quantifying over $\overline{v}$, we can move $\alpha_0,\alpha_1,\alpha_3$ and $\alpha_4$ one by one into the set of half integers. On one side by Fact \ref{P6} we then have infinitely many algebraic solutions whereas on the other side $X(\overline{\nu})$ (for any $\overline{\nu}$) can only have finitely many (Fact \ref{P2}, \ref{P3}, \ref{P4} and \ref{P5}). A contradiction.
\end{proof}

Finally one should note that the above arguments does not work when trying to show that generic $X_{II}(\alpha)$ are orthogonal to generic $X_{IV}(\alpha_1,\beta_1)$ and similarly that generic $X_{III}({v_{1},v_{2}})$ are orthogonal to generic $X_{V}(\alpha,\beta,\gamma)$. This is in part because the overall structure of each of the pairs of equations are very similar. 

\section*{Acknowledgement}
This research was mostly carried out at the University of Leeds, supported by an EPSRC Project Studentship and a School of Mathematics partial scholarship. I would like to thank Prof. Anand Pillay for his support and continuous encouragement. 

\end{document}